\documentclass{amsart}
\usepackage{amsmath}
\usepackage{amsfonts}
\usepackage{amssymb}
\usepackage[driverfallback=dvipdfm]{hyperref}
\usepackage{amssymb}
\usepackage{amsbsy}
\usepackage{amscd}
\usepackage{amsmath}
\usepackage{amsthm}
\usepackage{amsxtra}
\usepackage{latexsym}
\usepackage[mathscr]{eucal}
\usepackage{upgreek}
\usepackage{wasysym}
\usepackage[all,cmtip]{xy}

\usepackage{xcolor}
\definecolor{rouge}{rgb}{0.85,0.1,.4}
\definecolor{bleu}{rgb}{0.1,0.2,0.9}
\definecolor{violet}{rgb}{0.7,0,0.8}

\newcommand{\germ}{\mathfrak}
\newcommand{\cprime}{$'$}
\newcommand{\on}{\operatorname}
\newcommand{\+}{\mathop{\oplus}}
\renewcommand{\*}{{\otimes}}
\newcommand{\mc}{\mathcal}
\newcommand{\mf}{\mathfrak}

\newcommand{\fing}{\mf{g}}
\newcommand{\finq}{\mf{q}}
\newcommand{\affg}{\widehat{\mf{g}}}
\newcommand{\affq}{\widehat{\mf{q}}}

\newcommand{\Z}{\mathbb{Z}}
\newcommand{\C}{\mathbb{C}}
\newcommand{\N}{\mathbb{N}}

\newcommand{\W}{\mathscr{W}}
\newcommand{\ra}{\rightarrow}
\newcommand{\lam}{\lambda}

\DeclareMathOperator{\End}{End}

\DeclareMathOperator{\gr}{gr}

\DeclareMathOperator{\ad}{ad}

\theoremstyle{theorem}
\newtheorem{Th}{Theorem}[section]

\newtheorem{Pro}[Th]{Proposition}

\newtheorem{Lem}[Th]{Lemma}

\theoremstyle{remark}

\newtheorem{Rem}[Th]{Remark}

\newtheorem{Ex}[Th]{Example}

\begin{document}
\title[Quantising Mishchenko--Fomenko subalgebras for centralizers]{Quantizing 
Mishchenko--Fomenko subalgebras for centralizers
via affine $W$-algebras}
\author{Tomoyuki Arakawa}
\address{Research Institute for Mathematical Sciences, Kyoto University,
 Kyoto 606-8502 JAPAN}
\address{Dept of Mathematics, MIT, 182 Memorial Drive, Cambridge, MA 02139, USA}
\email{arakawa@kurims.kyoto-u.ac.jp}
\author{Alexander Premet}
\address{School of Mathematics, The University of Manchester,
Oxford Road, M13 9PL, UK}
\email{alexander.premet@manchester.ac.uk}
\thanks{T. Arakawa is partially supported by JSPS KAKENHI Grants
(\#25287004 and \#26610006).}
\thanks{A. Premet was supported by the Leverhulme Trust
(Grant RPG-2013-293)}
\begin{abstract}
We use affine $W$-algebras
to 
quantize Mishchenko--Fomenko subalgebras for centralizers 
of nilpotent elements in finite dimensional simple Lie algebras
under certain assumptions that are satisfied for all cases in type $\rm A$ and all minimal nilpotent cases outside type ${\rm E}_8$.
\end{abstract}
\maketitle\maketitle
\begin{center}
{\em Dedicated to Ernest Borisovich Vinberg for his 
$80$th birthday}
\end{center}

\vskip1ex

\vskip1ex

\vskip1ex

\vskip1ex

\vskip1ex

\vskip1ex 

\vskip1ex

\section{Introduction}
Let $G$ be a simple algebraic group over $\C$ and 
$\fing={\rm Lie}(G)$. Let
$e$ be a nilpotent element of $\fing$ and
$\{e,f,h\}\subset \fing$
an $\mf{sl}_2$-triple
containing
$e$.
Given $x\in \fing$ we write $\fing^x$ for the centralizer of $x$ in $\fing$. By Elashvili's conjecture, which is now a theorem, the index of $\fing^e$ equals 
$\ell:=\on{rk} \fing$; see \cite{C-M} and references therein. This means that the set of those linear functions on $\fing^e$ whose stabilizer in $\fing^e$ has dimension equal to $\ell$ is non-empty and Zariski open in $(\fing^e)^*$. The complement of this open subset in $(\fing^e)^*$ will be denoted by $(\fing^e)^*_{\on{sing}}$. 

We identify $\fing$ and $\fing^*$ as $G$-modules by using the Killing form of $\fing$. Under this identification, it was shown in \cite{PanPreYak06}
that for
$P\in S(\fing)^{\fing}$
the minimal degree component 
 ${}^{e\!} P$  of the restriction
$P|_{e+\fing^f}$ belongs to
$S(\fing^e)^{\fing^e}$.
A homogeneous generating system
$\{P_1,\dots,P_\ell\}$ of $S(\fing)^{\fing}$  is called 
{\em good} for $e$ if 
\begin{align*}
\sum_{i=1}^\ell \deg {}^{e\!}P_i=
(\dim \fing^e+\ell)/2.
\end{align*}
 \begin{Th}[\cite{PanPreYak06}]\label{Th:PPY}
Suppose $e$ admits a good generating system
$\{P_1,\dots, P_{\ell}\}$ in $S(\fing)^{\fing}$
and assume further that 
$(\fing^e)^*_{\on{sing}}$ has codimension $\geq 2$ in $(\fing^e)^*$.
Then $S(\fing^e)^{\fing^e}$ is a polynomial algebra in
variables
${}^{e\!}P_1,\dots, {}^{e\!}P_\ell$.
 \end{Th}

In \cite{PanPreYak06} it was also proved that
the assumptions of Theorem \ref{Th:PPY} hold for all nilpotent elements in simple Lie algebras of
type $\rm A$ and $\rm C$. We refer to
\cite{CM} for further results on polynomiality of 
$S(\fing^e)^{\fing^e}$.

\smallskip
Given $F\in S(\fing^e)$ and $\chi\in(\fing^e)^*$  we put \begin{align*}
 D_{\chi}F(x)\,:=\,\frac{d}{du}
F(x+ u\chi)|_{u=0}\quad\ \,\big(\forall\,x\in (\fing^e)^*\big),
\end{align*} so that
$F(x+ u\chi)=\sum_{i=0}^{\deg F}\frac{u^i}{i!} D_{\chi}^iF(x)$.
To each $\chi\in(\fing^e)^*$ we then attach
its {\it Mishchenko--Fomenko subalgebra}
$\bar{\mc{A}}_{e,\chi}$ of
$S(\fing^e)$.
Recall that the $\C$-algebra
$\bar{\mc{A}}_{e,\chi}$ is generated by 
the $\chi$-shifts $D_{\chi}^i (F)$ of all 
$F\in S(\fing^e)^{\fing^e}$.
It is well-known \cite{MisFom78} that
$\bar{\mc{A}}_{e,\chi}$ is a Poisson-commutative subalgebra 
of the Poisson--Lie algebra $S(\fing^e)$.

\begin{Th}[\cite{PanYak08}]\label{Th:PY}
Under the assumptions of Theorem \ref{Th:PPY} suppose further that 
$(\fing^e)^*_{\on{sing}}$ has codimension $\geq 3$ in $(\fing^e)^*$.
Then, for a  regular element 
$\chi$ 
of  $(\fing^e)^*$, the subalgebra
$\bar{\mc{A}}_{e,\chi}$ is
freely generated by all
$D_{\chi}^j ({}^{e\!}P_i)$ with
$i=1,\dots, \ell$ and $j=0,\dots \deg {}^{e\!}P_i-1$.
Moreover,
$\bar{\mc{A}}_{e,\chi}$
 is a maximal Poisson-commutative subalgebra of the Poisson--Lie algebra $S(\fing^e)$.
 \end{Th}

The canonical filtration of the universal enveloping algebra $U(\fing^e)$ induces that on any $\C$-subalgebra 
$\mc{A}$ of $U(\fing^e)$ and we write $\gr \mc{A}$
for the corresponding graded subalgebra of $S(\fing^e)
=\gr\, U(\fing^e)$. More precisely, $\gr\mc{A}$ is generated by the symbols in $S(\fing^e)$ of all $a\in\mc{A}$.
In this article we prove  the following:
 \begin{Th}\label{Th:Main}
Under the assumptions of Theorem \ref{Th:PPY}
there exists a 
commutative subalgebra $\mc{A}_{e,\chi}$ of $U(\fing^e)$
such that 
$\gr \mc{A}_{e,\chi} = \bar{\mc{A}}_{e,\chi}$. This subalgebra is freely generated by $(\dim\mf{g}^e+\ell)/2$ elements.
If $(\fing^e)^*_{\on{sing}}$ has codimension $\geq 3$ in $(\fing^e)^*$ then $\mc{A}_{e,\chi}$ is a maximal commutative subalgebra of $U(\fing^e)$.
 \end{Th}

If $e=0$ then $\fing^e=\fing$ and Theorem~\ref{Th:Main} becomes a classical result of Lie Theory. In this case, it has been proved by Tarasov \cite{Tar00} 
for $\fing=\mf{sl}_n$, 
by
Nazarov--Olshanski \cite{NazOls96} for classical Lie algebras, and 
by Rybnikov \cite{Ryb06},
Chervov--Falqui--Rybnikov \cite{CheFalRyb10}
and 
Feigin--Frenkel--Toledano-Laredo
\cite{FeiFreTol10} for any finite dimensional simple Lie algebra $\fing$.
Related results have also been obtained by Cherednik
\cite{Che}.

If $\mf{a}$ is a finite dimensional Lie algebra over $\C$ then the problem of quantizing 
Mishchenko--Fomenko subalgebras of $S(\mf{a})$, i.e. lifting them to commutative subalgebras of $U(\mf{a})$, was first raised in \cite{Vin91} and is sometimes referred to as {\it Vinberg's problem}. In this terminology, our paper shows that Vinberg's problem has a positive solution for
a large class of centralizers in simple Lie algebras.
It is worth remarking that Vinberg's
problem is wide open for an arbitrary Lie algebra $\mf{a}$.

In Section~\ref{sec:codim} we show that all conditions of
Theorem~\ref{Th:Main} hold in the case where $e$ is an element of the minimal nonzero nilpotent orbit $\mathcal{O}_{\rm min}\subset\mf{g}$ and $\mf{g}$ is not of type ${\rm E}_8$.
In Section \ref{sec:example}
we describe
$\mc{A}_{e,\chi}$ explicitly in the case where $e\in\mathcal{O}_{\rm min}$ and $\mf{g}$ is of type ${\rm A}_\ell$ with $\ell\in\{2,3\}$.

The main step
of the proof
of Theorem \ref{Th:Main}
consists in establishing a chiralization
of Theorem \ref{Th:PPY}.
Namely we prove the 
following result:
let $Z(V^{\kappa_{e,c}}(\fing^e))$ be the center of 
 the universal affine vertex algebra $V^{\kappa_{e,c}}(\fing^e)$
associated with $\fing^e$ at the critical level $\kappa_{e,c}$; see
Section \ref{section:main} for more detail.
There exists a natural filtration
of the vertex algebra $V^{\kappa_{e,c}}(\fing^e)$
such that $\gr V^{\kappa_{e,c}}(\fing^e)\cong 
S(\affg^e_-)$,
where $\affg^e_-=\fing^e[t^{-1}]t^{-1}:=\fing^e\otimes_\C t^{-1}\C[t^{-1}]$.
Given $a\in V^{\kappa_{e,c}}(\fing^e)$ we denote by $\sigma(a)$ 
its symbol in $S(\affg^e_-)$.
We identify $S(\fing^e)$ with a subring of $S(\affg^e_-)$ 
by using the algebra homomorphism which sends any $x\in\fing^e$ to
$xt^{-1}\in \affg^e_-$.
 \begin{Th}\label{Th:FF-theorem-for-centralizer}
Let $\{P_1,\dots, P_\ell\}$ be 
 a good generating system of $S(\fing)^{\fing}$  for a nilpotent element $e\in\fing$ and suppose that $(\fing^e)^*_{\on{sing}}$ has codimension $\geq 2$ in $(\fing^e)^*$.
Then there exist homogeneous elements
${}^{e\!}\hat P_1,\dots, {}^{e\!}\hat P_{\ell}$
in $Z(V^{\kappa_{e,c}}(\fing^e))$ 
such that
$\sigma(\hat P_i)={}^{e\!}P_i\in S(\fing^e)
\subset S(\affg^e_-)$,
and
$Z(V^{\kappa_{e,c}}(\fing^e))$ is a polynomial  algebra 
in variables
 $T^{j}({}^{e\!}\hat P_i)$ with 
  $i=1,\dots,\ell$ and $j\ge 0$,
where $T$ is the translation operator of 
$V^{\kappa_{e,c}}(\fing^e)$.
 \end{Th}
Having established  Theorem \ref{Th:FF-theorem-for-centralizer} we  deduce
 Theorem \ref{Th:Main}  by
 applying 
 the method which
 Rybnikov \cite{Ryb06} used in the case where $e=0$. It should be mentioned that in this case 
Theorem \ref{Th:FF-theorem-for-centralizer}
is a celebrated result of Feigin and Frenkel \cite{FeiFre92}.
In order to prove Theorem \ref{Th:FF-theorem-for-centralizer} in our situation we use the affine 
$W$-algebra 
$\W^{\kappa_c}(\fing,f)$ associated with $(\fing,f)$ at the
critical level $\on{\kappa_{c}}$.

\noindent
{\bf Acknowledgement.} The work of this paper has started during the second author's visit to Kyoto University. He would like to thank RIMS (Kyoto) for its hospitality and excellent working conditions. The authors are grateful to the anonymous referee for pointing out a gap in an earlier version of this paper and important suggestions.
\section{Rybnikov's construction and vertex algebras}\label{sec1}
Let $\finq$ be a finite dimensional Lie algebra over $\C$ and let
$\kappa(\,\cdot\,,\,\cdot\,)$ be a symmetric invariant bilinear form of 
$\finq$.
The associated affine Kac-Moody algebra
$\affq_{\kappa}$  is a central extension of the
Lie algebra $\finq((t))=\finq\* \C(\!(t)\!)$ by the one-dimensional center
$\C \mathbf{1}$:
\begin{align*}
 0\ra \C \mathbf{1}\ra \affq_{\kappa}\ra \finq(\!(t)\!)\ra 0.
\end{align*}
The commutation relations in $\affq_{\kappa}$ are given by
\begin{align*}
 [x\* f,y\* g]\,:=\,[x,y]\* fg -\kappa(x,y)\on{Res}_{t=0}(f dg) \mathbf{1}
\end{align*}
for all $x,y\in \finq$ and
$f,g\in \C(\!(t)\!)$.
Set 
\begin{align*}
\affq_+=
\finq[[t]]\+ \C \mathbf{1},\quad
\affq_-=\finq[t^{-1}]t^{-1}\subset \affq_{\kappa},
\end{align*}
so that
 $\affq_{\kappa}=\affq_-\+ \affq_+$.

We regard $S(\finq)$ as a  subalgebra
of 
$S(\affq_-)$ via the embedding $x\mapsto xt^{-1}$,
 $x\in \finq$.
We also identify $S(\affq_-)$ with the function algebra $\C[\finq_{\infty}^*]$  on the
infinite jet scheme $\finq_{\infty}^*$
 of the affine space
 $\finq^*$.
Then our embedding $S(\finq)\hookrightarrow S(\affq_-)$
gives rise to an 
inclusion
$\C[\finq^*]\hookrightarrow\C[\finq_{\infty}^*]$.
The Poisson algebra structure of $S(\finq)=\C[\finq^*]$ gives rise
\cite{Ara12}
to a
Poisson
{\it vertex} algbera structure on
$S(\affq_-)=\C[\finq_\infty^*]$.
The translation operator $T$ 
on the Poisson vertex  algebra
$S(\affq_-)$ is the derivation
of the $\C$-algebra $S(\affg_-)$
defined by the rule
\begin{align*}
 T x_{(-m)}=m x_{(-m-1)},
\quad T1=0,
\end{align*}
where $x_{(m)}=xt^m$,
$x\in \finq$.
Let
$Z(S(\affq_-))$
be the center
of the Poisson vertex algebra
$S(\affq_-)$.
Identifying  $\affq_-$ with $\affq/\affq_+$ we have that
\begin{align*}
Z(S(\affq_-))\,=\,
S(\affq_-)^{\affq_+}.
\end{align*}

For $\chi\in \finq^*$
and $u\in \C^*$, we
define a ring homomorphism
\begin{align*}
\bar {\Phi}_{\chi}(?,u): S(\affq_-)\ra 
S(\finq),
\quad 
\end{align*}
by setting
\begin{align*}
\bar \Phi_{\chi}( x_{(-m)},u)= u^{-m}x+\delta_{m,1}\chi(x)
 \quad\text{for all $x\in \finq$ and $m\in \N$.}
\end{align*}
This definition implies that
\begin{align}
\bar \Phi_{\chi}(P,u)(\lam)=u^{-\deg P}P(\lam+u\chi)
\label{eq:bar-derivation}
\end{align}
for any homogeneous element $P\in S(\finq)\subset S(\affq_-)$ and any
$\lam\in \finq^*$. Furthermore,
\begin{align}\label{TT-shift}
 \bar \Phi_{\chi}(T a, u)=-\frac{d}{du}\bar \Phi_{\chi}(a,u)
\end{align}
for all $a\in S(\affq_-)$.
We denote by
$\bar \Phi_{\chi,n}(a)\in S(\finq)$ the coefficient of
$u^{-n}$ in $\bar \Phi_{\chi}(a,u)$,
so that 
\begin{align*}
 \bar \Phi_{\chi}(a,u)=\sum_{n\geq 0}\bar \Phi_{\chi,n}(a)u^{-n}.
\end{align*}
For a homogeneous element $P$ of $S(\fing)$ we have that
\begin{align}
 \bar \Phi_{\chi,n}(P)=\begin{cases}
			\frac{u^{\deg P-n}}{(\deg P-n)!}D_{\chi}^{\deg
			P-n}
			(P)& \mbox{ if}\ n\leq \deg P;\\ 0&\ \text{else}.
		       \end{cases}
\label{eq:formula-for-bar-P}
\end{align}

\begin{Lem}[\cite{Ryb06,FeiFreTol10}]
\label{Lem:independent1}
The following are true:

\begin{enumerate}
 \item   For any $T$-invariant subspace $U$  of  $S(\affq_-)$ 
the image $\Phi_{\chi}(U,u)$ is independent of the choice of $u\in \C^*$
 and 
spanned by the elements $\Phi_{\chi,n}(a)$ with $a\in U$.

\smallskip

\item  Let $A$ be a subalgebra of $S(\finq)^{\finq}$ and let
$A_{\infty}$ be the ($T$-invariant) subalgebra of $S(\affq_-)$ generated  by
$A$.
Then
$\bar \Phi_{\chi}(A_{\infty})$ is the 
Mishchenko-Fomenko subalgebra 
of $S(\finq)$ generated by 
the $\chi$-shifts of $P$
for all $P\in A$.
\end{enumerate}
\end{Lem}

Let
\begin{align*}
V^\kappa(\finq)
=U(\affq_{\kappa})\*_{U(\affq_+)}\C,
\end{align*}
where 
$\C$ is considered as a 
 $\affq_+$-module on which 
$\finq[[t]]$ acts trivially and 
 $\mathbf{1}$ acts as the
identity.
Note that
\begin{align*}
V^{\kappa}(\finq)\cong U(\affq_-)
\end{align*}
as vector spaces.
As is well-known (see \cite{Kac98,FreBen04}),
$V^{\kappa}(\finq)$ is naturally a vertex algebra.
The  state-field correspondence
\begin{align*}
 V^{\kappa}(\finq)
\longrightarrow\, \End V^{\kappa}(\finq)[[z,z^{-1}]],
\quad a\mapsto a(z)=\sum_{n\in\Z }a_{(n)}z^{-n-1},
\end{align*}
is uniquely determined by the conditions
\begin{align*}
 (\mathbb{I})(z)=\on{id}_{V^{\kappa}(\finq)},\quad 
(xt^{-1}\mathbb{I})(z)=x(z):=\sum_{n\in \Z} x_{(n)} z^{-n-1} \text{ for }x\in \finq,
\end{align*}
where $\mathbb{I}=1\* 1\in V^{\kappa}(\finq)$ and $x_{(n)}=xt^n$.
The translation operator
$T$ is an endomorphism of $V^{\kappa}(\finq)$
such that
$T\mathbb{I}=0$ and
$[T, x(z)]=\frac{d}{dz}x(z)$
for all $x\in \fing$.

Given a  vertex algebra $V$
we write  $V=E^0V\supset E^1V\supset E^2V\supset \dots $ for the {\it Li filtration} of $V$; see \cite{Li05} and \cite{Ara12}. Recall that $E^pV^{\kappa}(\finq)$ is spanned by all
$$(x_1t^{-n_1})\cdots (x_rt^{-n_r})\mathbb{I}\ \mbox{ with  }\ x_i\in\finq\ \mbox{ and }\, 
-r+\sum_{i=1}^r n_i\ge p.$$ 
The associated graded space
$\gr V=\bigoplus_p E^pV /E^{p+1}V$ is 
naturally a Poisson vertex algebra.
The Poisson vertex algebra structure on $\gr V$
induces a genuine Poisson algebra structure on Zhu's $C_2$-algebra $R_V:=V/E^1 V\subset \gr V$. Moreover, $R_V$ generates
$\gr V$ as a differential graded algebra; see \cite{Li05}.

Setting $\on{deg}xt^{-n}=n$ for all $x\in\finq$
and $n\ge 0$ gives
$U(\affq_-)=\bigoplus_{j\in \Z_{\geq 0}}U(\affq_-)[j]$ a graded algebra structure.
This, in turn, induces a grading on 
$V^{\kappa}(\mf{q})$:
\begin{align}
V^{\kappa}(\finq)=\bigoplus\limits_{j\in \Z_{\geq 0}}V^{\kappa}(\finq)[j],
\label{eq:grading}
\end{align}
where $V^{\kappa}(\finq)[j]=(U(\affq_-)[j])\mathbb{I}$.
Let 
$\{U_i(\affq_-)\,|\,\, i\ge 0\}$ be the 
canonical filtration
of $U(\affq)$ and put
$U_i(\affq_-)[j]=U_i(\affq_-)\cap U(\affq_-)[j]$.
The above then shows that
$$(E^pV^{\kappa}(\finq))[j]=(U_{j-p}(\affq_-)[j])\mathbb{I},$$
where
$(E^pV^{\kappa}(\finq))[j]=E^p V^{\kappa}(\finq)\cap V^{\kappa}(\finq)[j]$.
It follows that
\begin{align}
\gr V^{\kappa}(\mf{q})\cong S(\widehat{\mf{q}}_-)\ \,\,
\mbox{ and }\ \, 
R_{V^\kappa(\finq)}\cong S(\mf{q}).
\label{eq:case-of-affine}
\end{align}
More precisely, $R_{V^\kappa(\finq)}$ coincides with the subalgebra of $S(\widehat{\mf{q}}_-)$ generated by all
$xt^{-1}$ with $x\in\finq$.

Given a vertex algebra
$V$ we denote by
$Z(V)$ the {\it center} of $V$, so that
\begin{eqnarray*}
Z(V)&=& \{z\in V\,|\,\, [z_{(m)}, a_{(n)}]=0
\ \text{for all }a\in V \text { and } m,n\in \Z\}
\\
&=&\{z\in V\,|\,\, a_{(n)}z=0\text{ for all }a\in V \text{ and } n\geq 0\}.
\end{eqnarray*}
Here the last equality follows from
 the commutator formula
\begin{align*}
 [a_{(m)},b_{(n)}]=\sum_{i\geq 0}\begin{pmatrix}
				  m\\ i
				 \end{pmatrix}(a_{(i)}b)_{(m+n-i)}.
\end{align*}
Since $Z(V)$ is a {\it commutative} vertex subalgebra of $V$, it is a differential algebra, i.e. a unital commutative $\C$-algebra equipped with a
derivation; see \cite{Bor86}. Specifically, the
multiplication in $Z(V)$ is given by $a\cdot b:=a_{(-1)}b$ for all $a,b\in Z(A)$. The translation operator of $V$ preserves $Z(A)$ and acts on this algebra as a derivation.

When
$V=V^{\kappa}(\finq)$
we have that
\begin{align*}
 Z(V^{\kappa}(\finq))
=
V^{\kappa}(\finq)^{\affq_+}\cong\, \End_{\affq_{\kappa}}(V^{\kappa}(\finq))^{\on{opp}}
\end{align*}
as $\C$-algebras.
We will often regard
$ Z(V^{\kappa}(\finq))$ as a commutative subalgebra 
of $U(\affq_-)$ 
via the above-mentioned identification $V^{\kappa}(\finq)=U(\affq_-)$.

Given $\chi\in \finq^*$
and $u\in \C^*$ we
define an algebra homomorphism
\begin{align*}
\Phi_{\chi}(?, u):  U(\affq_-)\ra U(\finq),\quad 
a\mapsto \Phi_{\chi}(a,u),
\end{align*}
by the rule
\begin{align*}
\Phi_{\chi}(x_{(-m)},u)= u^{-m}x +\delta_{m,1}\chi(x)
\quad
\text{for all }
 x\in \finq \,\text{ and } m\in \Z.
\end{align*}
It is immediate from the definition  that
\begin{align}\label{T-shift}
 \Phi_{\chi}(Ta,u)=-\frac{d}{d u}\Phi_{\chi}(a,u)
\end{align}
for all $a\in V^\kappa(\finq)=U(\affq_-)$.
We denote by 
$ \Phi_{\chi,n}(a)\in U(\finq)$ the coefficient of
$u^{-n}$ in $ \Phi_{\chi}(a,u)$,
so that 
\begin{align*}
\Phi_{\chi}(a,u)=\sum_{n\geq 0}\Phi_{\chi,n}(a)u^{-n}.
\end{align*}
The following result can be found in \cite[Sect.~2.5]{FeiFreTol10}.
\begin{Lem}\label{Lem:FFTL}
 For any $T$-invariant subspace $U$  of  $U(\affq_-)$ 
the image $\Phi_{\chi}(U)$ is independent of the choice of $u\in \C^*$
 and 
spanned by $\Phi_{\chi,n}(a)$ with $a\in U$.
\end{Lem}
We let 
$\mc{A}_{\chi}$ denote
the image of the subalgebra $Z(V^{\kappa}(\finq))$ of $U(\affq_-)$ 
under the homomorphism $\Phi(?,u)$. This is a commutative $\C$-subalgebra of $U(\finq)$. 
\begin{Lem}\label{Lem:symbol}
Let $a\in V^{\kappa}(\finq)[d]$ be such that its image  $\bar{a}$ 
in $R_{V^\kappa(\finq)} \cong S(\finq)$ has the property that $D_\chi^{d-i}\big(\bar{a})\ne 0$ for some $i\le d$.  Then
the symbol of $\Phi_{\chi,i}(a)$ in $S(\finq)={\rm gr}\,U(\finq)$ 
equals $\bar{\Phi}_{\chi,i}(\bar{a})$. 
\end{Lem}
\begin{proof} The statement of the lemma is a reformulation (in a different notation) of \cite[Lemma~3.13]{FeiFreTol10} and the proof in {\it loc.\,cit.} applies in our situation without any changes.
\end{proof}
\section{Affine $W$-algebras, the
Feigin--Frenkel center 
and quantization of Mishchenko--Fomenko subalgebras}
\label{section:main}
Since we are going to apply the methods outlined in 
Section~\ref{sec1} to the case where $\finq=\fing^e$
we shall require an extension of \cite[Th\'eor\`em~4.5(ii)]{RaTau92} valid under our assumptions on the nilpotent element $e\in\fing$. 

We first fix some notation. Let $m$ be a nonnegative integer and write $T$ for the image of $t$ in the truncated polynomial ring $\mc{O}_m:=\C[t]/(t^m)$. Set $\mf{q}:=\mf{g}^e$ and write $\mf{q}_{m}$ for the Lie algebra
$\mf{q}\otimes \mc{O}_m$. As a vector space, $\mf{q}_m=\bigoplus_{i=0}^{m-1}\mf{q}T^i$, 
where $T$ stands for the image of $t$ in $\mc{O}_m$, and 
$[x T^i,yT^j]=[x,y] T^{i+j}$ for all
$x,y\in\mf{q}$ and $i,j\in\Z_{\ge 0}$ (to ease notation we omit all tensor product symbols).
Following the conventions of \cite{RaTau92} we identify  $\mf{q}_m^*=\bigoplus_{i=0}^{m-1}\mf{q}^*T^i$ with the direct product of $m$ copies of $\mf{q}^*$. Specifically, the $m$-tuple
 $(\psi_0,\ldots, \psi_{m-1})$ with $\psi_0,\ldots, \psi_{m-1}\in\mf{q}^*$ identifies with the linear function $\psi$ on $\mf{q}_m$ such that $\psi(x T^i)=\psi_i(x)$ for all $x\in\mf{q}$ and $0\le i\le m-1$.
 
Recall that the {\it index} of a finite dimensional Lie algebra $\mf{a}$, denoted ${\rm ind}\,\mf{a}$,  is defined as the smallest dimension of the coadjoint stabilizers $\mf{a}^\chi$ where $\chi$ runs through the linear functions on $\mf{a}$. We denote by $\mf{a}^*_{\rm{sing}}$ the Zariski closed, conical subset of $\mf{a}^*$ consisting of all $\chi\in\mf{a}^*$ with
$\dim\mf{a}^\chi>{\rm ind}\,\mf{a}$. Since ${\rm ind}\,\finq={\rm ind}\, \fing=\ell$, it follows from
\cite[1.6]{RaTau92} that ${\rm ind}\,\mf{q}_m=m\ell$ and 
$$\psi=(\psi_0,\ldots,\psi_{m-1})\in\mf{q}_m^*\setminus(\mf{q}_m)^*_{\rm sing}  
 \iff \psi_{m-1}\in\mf{q}\setminus(\mf{q}^*)_{\rm sing}.
 $$
Each polynomial function $P\in S(\mf{q})$ on $\mf {g}^*$ gives rise to $m$ polynomial functions  $P_0,\ldots, P_{m-1}\in S(\mf{q}_m)$ on $\mf{q}_m^*$ by the rule
$$P\big({\textstyle \sum}_{i=0}^{m-1}\psi_it^i\big)\,=\,{\textstyle \sum}_{i\ge 0} P_i(\psi_0,\ldots,\psi_{m-1})t^i\qquad \big(\forall\,(\psi_0,\ldots,\psi_{m-1})\in\mf{q}_m^*\big)$$ 
(we ignore the coefficients $P_i(\psi_0,\ldots,\psi_{m-1})$ with $i\ge m$).
It is worth mentioning that if $P\in S^r(\mf{q})$ then $P_i\in S^r(\mf{q}_m)$ for all $0\le i\le m-1$. By \cite[Lemma~3.5]{RaTau92},
all $P_i$ with $0\le i\le m-1$ lie in $S(\mf{q}_m)^{\mf{q}_m}$ whenever $P\in S(\mf{q})^{\mf{q}}$.
\begin{Pro}\label{Pro:inv}
Let 
$\mf{q}=\mf{g}^e$ and suppose the pair $(\mf{g},e)$ satisfies the conditions of Theorem~\ref{Th:PPY}.
If $\{P_1,\ldots, P_\ell\}\subset S(\mf{g})^\fing$ is a good generating system for $e$ then the elements $(^e\!P_i)_j$ with 
$1\le j\le \ell$ and $0\le j\le m-1$ are algebraically independent and generate
 the $\C$-algebra $S(\mf{q}_m)^{\mf{q}_m}$.
\end{Pro}
\begin{proof}
Since $^e\!P_i\in S(\mf{q})^{\mf{q}}$ the preceding remark shows that $(^e\!P_i)_j\in S(\mf{q}_m)^{\mf{q}_m}$ for all admissible $i$ and $j$. 
By \cite[Lemma~3.3(i)]{RaTau92}, the differentials
${\rm d}\,(^e\!P_i)_j$ with $1\le i\le \ell$ 
and $0\le j\le m-1$ are linearly independent at 
$\psi=(\psi_0,\ldots,\psi_{m-1})\in\mf{q}_m^*$ if and only if
the differentials ${\rm d}\,^e\!P_i$ are linearly independent at $\psi_{m-1}\in\mf{q}^*$. 
 Since 
$(\mf{q}^*)_{\rm sing}$ coincides with 
the Jacobian locus $\mathcal{J}(^e\!P_1,\ldots,^e\!P_\ell)$ 
by \cite[Theorem~2.1]{PanPreYak06}, the latter has codimension $\ge 2$ in $\mf{q}^*$ (here we use our assumption on $e$). The above discussion then yields that the same holds for the Jacobian 
locus of the $(^e\!P_i)_j$'s in $\mf{q}_m^*$. As a consequence, these elements are algebraically 
independent in $S(\mf{q}_m)$ and satisfy the conditions 
of \cite[Theorem~1.1]{PanPreYak06} which is an extended characteristic zero version of \cite[Theorem~5.4]{Sk02}. Applying this result shows that the $\C$-subalgebra, $R$, of $S(\mf{q}_m)^{\mf{q}_m}$ generated by the $m\ell={\rm ind}\,\mf{q}_m$ elements $(^e\!P_i)_j$ with $1\le i\le \ell$ and $0\le j\le m-1$
is algebraically closed in $S(\mf{q}_m)$. On the other hand, it is well known that
${\rm ind}\,\mf{q}_m$ coincides with the transcendence degree of the field of invariants $\C(\mf{q}_m^*)^{\mf{q}_m}$.
It follows that the field of fractions of $S(\mf{q}_m)^{\mf{q}_m}$ is algebraic over the subfield  generated by $R$. As $R$ is algebraically closed in $S(\mf{q}_m)^{\mf{q}_m}$ we now deduce that $R=S(\mf{q}_m)^{\mf{q}_m}$ and the proposition follows.
\end{proof}
As an immediate consequence we obtain the following:
\begin{Th}\label{Th:associated-graded}
Suppose all assumptions of Theorem \ref{Th:PPY}
hold and
let $\{P_1,\dots, P_\ell\}\subset S(\fing)^{\fing} $  be 
 a good generating system for $e$.
 Then 
$S(\affg^e_-)^{\affg^{e}_+}$ is a polynomial algebra
in variables 
 $T^j ({}^{e\!}P_i)$ where
$i=1,\dots, \ell$ and
$j\in\Z_{\ge 0}$.
\end{Th}
\begin{proof}
Let $\finq=\fing^e$  and write $\mf{q}_\infty$ for the Lie algebra $\mf{q}\otimes \C[t]=\finq[[t]]$. Since we identify
$\affq_-$ with $\affq/\affq_+$ the invariant ring
$S(\affq_-)^{\affq_+}$ gets identified with $S(\mf{q}_\infty)^{\mf{q}_\infty}$.
Since the conditions of Proposition~\ref{Pro:inv}
are satisfied for all Lie algebras $\finq_{(m)}=\finq[[t]]/t^m\finq[[t]]$ with $m\ge 0$
it is straightforward to see that
each invariant algebra
$\C[(\fing^e)_{(m)})^*]^{\fing^e_{(m)}}$ is freely generated by the elements $T^j(^{e\!}P_i)$ with
$1\le i\le \ell$ and $0\le j\le m-1$.
Note that $(\mf{q}_{(m)})^*$ is the $m$-th jet scheme of the affine space $\mf{q}^*$ and {\it all} polynomial functions
$T^j(^{e\!}P_i)\in\C[(\finq)^*_{\infty}]$ are invariant under the natural action of the group $G^e\big(\C[t]\big)$, where $G^e=Z_G(e)$. It follows that $T^j(^{e\!}P_i)\in\C[(\finq)^*_{\infty}]^{\finq_\infty}$ for all $i\le \ell$ and $j\ge 0$. In conjunction with the above this shows that
$$ S(\affq_-)^{\affq_+}=\,\C[(\finq)^*_{\infty}]^{\finq_{\infty}}=\,\varinjlim_m\,
\C[(\finq)^*_{(m)}]^{\finq_{(m)}}$$
is a polynomial algebra
in $T^j ({}^{e\!}P_i)$ where
$i=1,\dots, \ell$ and
$j\in\Z_{\ge 0}$.
\end{proof}

We now fix a good grading \cite{KacRoaWak03,ElaKac05}
\begin{align*}
 \fing=\bigoplus_{j\in \Z}\fing(i)
\end{align*}
for $e$.
Then $\fing^e=\bigoplus_{j\geq 0}\fing^e(j)$,
where $\fing^e(i)=\fing^e\cap \fing(i)$.
Let $\kappa_{e,c}$ be the invariant bilinear form
of $\fing^e$ defined by
\begin{align*}
 \kappa_{e,c}(x,y)=\begin{cases}
		    -\frac{1}{2}\on{tr}_{\fing_0}(\ad x\ad
		    y)-\frac{1}{2}\on{tr}_{\fing_1}(\ad x \ad y)&\text{for
		    }
x,y\in \fing^e(0),\\
0&\text{else},
		   \end{cases}
\end{align*}
for all homogeneous
$x,y\in \fing^e$.

Let 
$\W^{\kappa_c}(\fing,f)$ be the affine
$W$-algebra \cite{KacRoaWak03}
associated with
$(\fing,e)$ at the
{\em critical level} $\on{\kappa_{c}}$.
Here
\begin{align*}
\kappa_{c}(x,y)=-\frac{1}{2}\on{tr}_{\fing}(\ad x \ad
 y)\quad\text{for }x,y\in \fing.
\end{align*}
The following assertion was proved in \cite{KacWak04}
(see also \cite[Theorem 5.5.1]{Ara05}).
\begin{Th}[{\cite{KacWak04}}]\label{Th:loop-filtration}
 There exists a filtration
$F_{\bullet}\W^{\kappa_c}(\fing,e)$
of the vertex algebra  $\W^{\kappa_c}(\fing,e)$ such that
\begin{align*}
 \gr_F \W^{\kappa_c}(\fing,e)\,\cong\, 
V^{\kappa_{e,c}}(\mf{g}^e)
\end{align*}
as vertex algebras.
\end{Th}
This result enables us to attach to every element $a\in \W^{\kappa_c}(\fing,e)$ its $F$-{\it symbol} 
$\sigma_F(a)\in V^{\kappa_{e,c}}(\mf{g}^e)$. 
Let
$\mc{S}_e:=e+\mf{g}^f$, the Slodowy slice to the adjoint $G$-orbit of $e$.
\begin{Th}[{\cite[Theorem 4.4.7]{Ara09b}}]\label{Th:grW}
 There is a natural isomorphism
\begin{align*}
 \gr \W^{\kappa_c}
(\fing,f)\,\cong\, \C[(\mathcal{S}_e)_{\infty}],
\end{align*}
where
$ \gr \W^{\kappa_c}(\fing,f)$ denotes the graded Poisson vertex algebra
associated with the Li filtration of 
$\W^{\kappa_c}(\fing,f)$.
\end{Th}
Let $\finq=\fing^e$. The $F$-filtration also
induces a filtration of
the associated graded
vertex Poisson algebra
$\gr \W^{\kappa_c}(\fing,e)$,
which we identify with $ \C[(\mathcal{S}_e)_{\infty}]$ by Theorem \ref{Th:grW}.
Using the BRST realization of $ \C[(\mathcal{S}_e)_{\infty}]$ (\cite{Ara09b}),
we  find in the same  way as Theorem \ref{Th:loop-filtration}
that \begin{align}
\gr_F (\C[(\mathcal{S}_e)_{\infty}])\cong \C[(\mf{q}^*)_{\infty}],
\label{eq:gr-f-for-Poisson},
\end{align}
which restricts to an isomorphism
$
\gr_F \C[\mc{S}_e]\cong  \C[\mf{q}^*]
$.
For $a\in \C[(\mathcal{S}_e)_{\infty}]$, let $\sigma_F(a)\in \C[(\mf{q}^*)_{\infty}]$  denote its $F$-symbol.
It will be crucial in what follows that 
for $a\in \C[\mathcal{S}_e]\subset \C[(\mathcal{S}_e)_{\infty}]$ the $F$-symbol
$\sigma_F(a)\in \C[\mf{q}^*]$ coincides with the initial term of $a $ in the sense of \cite{PanPreYak06}.

Theorem \ref{Th:loop-filtration} in conjunction with
\eqref{eq:case-of-affine}
and
\eqref{eq:gr-f-for-Poisson}
implies that 
 \begin{align}
\sigma_F(\sigma(a))=\sigma(\sigma_F(a))\ \mbox{ for all }\,\, a\in  \W^{\kappa_c}(\fing,e).
\label{eq:two-filtration-commutes}
\end{align}

\begin{Th}[\cite{A11}]\label{Th:center-of-W}
There is a natural  isomorphism
of vertex algebras 
\begin{align*}
Z(V^{\kappa_c}(\fing))\,\stackrel{\sim}{\longrightarrow}\, Z(\W^{\kappa_c}(\fing,e)),
\end{align*}
which induces an
embedding
\begin{align*}\,
 S(\affg_-)^{\affg_+} \hookrightarrow  \C[(\mc{S}_e)_{\infty}],
\quad P\mapsto P|_{(\mc{S}_e)_{\infty}},
\end{align*}
of associated graded vertex Poisson algebras.
\end{Th}
The image of $\hat{P}_i\in Z(V^{\kappa_c}(\fing))$ 
in $Z(\W^{\kappa_c}(\fing,e))$ will also be denoted by 
$\hat{P}_i$. 
The following result is well-known.
\begin{Th}[Feigin and Frenkel \cite{FeiFre92,Fre05}]
Each set $\{P_1,\dots,P_{\ell}\}$ of homogeneous
generators of $S(\fing)^{\fing}$ admits a lift
$\{\hat P_1,\dots, \hat P_{\ell}\}\subset 
Z(V^{\kappa_c}(\fing))$ such that
the symbol of $\hat{P_i}$ equals to $P_i\in S(\fing)^{\fing}
\subset S(\affg_-)^{\affg_+}$ for each $i\le \ell$.
Furthermore,
$Z(V^{\kappa_c}(\fing))$ is a polynomial algebra
in variables  $T^i \hat{P_j}$ with $i\in\Z_{\ge 0}$ and $j=1,\dots,\ell$.
\end{Th}
We may assume that 
each $\hat P_i$ is homogeneous with respect to  the grading 
\eqref{eq:grading}.

\begin{proof}[Proof of Theorem \ref{Th:FF-theorem-for-centralizer}]
If $a\in Z(\W^{\kappa_c}(\fing,e))$ then $\sigma_F(a)\in Z(V^{\kappa_{e,c}}(\fing^e))$.
We set
$${}^{e\!}\widehat{P_i}:=\sigma_F(\widehat{P_i})\in Z(V^{\kappa_{e,c}}(\fing^e)).$$
Then, by \eqref{eq:two-filtration-commutes},
\begin{align*}
\sigma({}^{e\!}\widehat{P_i})=\sigma(\sigma_F(\widehat{P}_i))=\sigma_F(\sigma(\widehat{P}_i))
=\sigma_F(P_i)={}^e P_i
\end{align*}
for each $i=1,\dots,\ell$.
Thus,
${}^{e\!}\widehat{P_i}\in  Z(V^{\kappa_{e,c}}(\fing^e))$ is a lift of ${}^e P_i\in S(\affg^e_-)^{\affg^e_+}$.
We have
$\sigma(T^j({}^{e\!}\widehat{P_i}))=T^j({}^e P_i)
$ for all $j\geq 0$.
The above-mentioned identification $\gr V^{\kappa_{e,c}}(\fing^e)=
S(\affg^e_-)$ 
then induces
an embedding
$$\gr Z(V^{\kappa_{e,c}}(\fing^e))
=\gr (V^{\kappa_{e,c}}(\fing^e)^{\affg_+})\hookrightarrow 
S(\affg^e_-)^{\affg^e_+}.$$
 In view of Theorems~\ref{Th:associated-graded} the assertion follows.
\end{proof}
\begin{proof}[Proof of Theorem \ref{Th:Main}]
Since $\Phi_\chi$ is an algebra homomorphism, setting in Lemma~\ref{Lem:FFTL} $U=Z(V^{\kappa_{e,c}}(\fing^e))$ 
and applying Theorem \ref{Th:FF-theorem-for-centralizer}
one observes that the subalgebra 
 $\mc{A}_{e,\chi}$ of $U(\fing^e)$ generated by 
all $\Phi_{\chi,n}(^e\!\widehat{P_i})
$ with 
$1\le i\le \ell$ and
$0\le n\le \deg ^e\!P_i-1$ is commutative. Since $\fing^e$ satisfies the conditions of Theorem~\ref{Th:PPY} and $\chi$ is regular in $(\fing^e)^*$, it follows from \cite[Theorem~3.2(i)]{PanYak08} that
the $\chi$-shifts $D_\chi^j(^e\!P_i)$ with $1\le i\le \ell$ and $0\le j\le \deg ^e\!P_i-1$ are algebraically independent in $S(\fing^e)$ and hence nonzero. 
Applying Lemma~\ref{Lem:symbol} 
now yields that the symbol of every 
$\Phi_{\chi,n}({^e}\!\widehat{P_i})$ in $S(\fing^e)$ equals
$\bar \Phi_{\chi,n}(^e\!P_i)$.  
In view of \eqref{TT-shift}, \eqref{T-shift} and \eqref{eq:formula-for-bar-P} this means that the 
Poisson-commutative subalgebra $\gr \mc{A}_{e,\chi}$
contains $\bar{\mc{A}}_{e,\chi}$. 

Conversely, if $a\in \mc{A}_{e,\chi}$ then 
Lemma~\ref{Lem:symbol} combined with \cite[Theorem~3.2(i)]{PanYak08} yields that the symbol of $a$ in $S(\fing^e)$ lies in the subalgebra generated by $D_\chi^j(^e\!P_i)$;
see \cite[Theorem~3.14]{FeiFreTol10} for a similar argument. Hence $\gr \mc{A}_{e,\chi}=\bar{\mc{A}}_{e,\chi}$.
If $(\fing^e)^*_{\on{sing}}$ has codimension $\ge 3$ in $(\fing^e)^*$ then Theorem~\ref{Th:PY} says that $\bar{\mc{A}}_{e,\chi}$ is a maximal 
Poisson-commutative subalgebra of $S(\fing^e)$. This implies that $\gr \mc{A}_{e,\chi}$ is a maximal commutative subalgebra of $U(\fing^e)$, completing the proof.
\end{proof}
\begin{Rem}\label{RR}
We see that if $e$ satisfies the conditions of Theorem~\ref{Th:PPY} and $\chi$ is a regular linear function on $\fing^e$ then the commutative subalgebra $\mc{A}_{e,\chi}$ of $U(\fing^e)$ is a quantization of the Mishchenko--Fomenko subalgebra $\bar{\mc{A}}_{e,\chi}$. So Vinberg's problem has a positive solution in this case. However, when
$(\fing^e)^*_{\on{sing}}$ has codimension $2$ in $(\fing^e)^*$ the maximality of the commutative subalgebra $\mc{A}_{e,\chi}$ in $U(\fing^e)$ is not guaranteed. 
\end{Rem}
Remark~\ref{RR} brings to our attention the problem of classifying those nilpotent centralisers $\fing^e$ for which $(\fing^e)^*_{\on{sing}}$ has codimension $\ge 3$ in $(\fing^e)^*$. In the next section we show that this property holds in the case where $\ell\ge 2$ and $e$ lies in the minimal nonzero nilpotent orbit of $\fing$.
\section{The singular locus in the minimal nilpotent case}
\label{sec:codim} 
If $G$ is a group of type ${\rm A}$ then it follows from \cite[Theorem~5.4]{Yak09} that $(\mf{g}^e)^*_{\rm sing}$ has codimension $\ge 3$ in $(\mf{g}^e)^*$ for any non-regular nilpotent element $e\in\mf{g}$. On the other hand, the example in \cite[3.9]{PanPreYak06} shows that outside type ${\rm A}$ there exist nilpotent elements $e\in\mf{g}$
for which $(\mf{g}^e)^*_{\rm sing}$ has codimension $1$ in $(\mf{g}^e)^*$.
Let $\mathcal{O}_{\rm min}$ denote the minimal nonzero nilpotent 
$G$-orbit in $\mf{g}$ (it consists of all nonzero
elements $e\in\mf{g}$ such that $[e,[e,\mf{g}]]=\C e$).
In this section we are going to prove the following:
\begin{Th}\label{Th:codim3} If $\mf{g}$ is a finite dimensional simple Lie algebra of rank $>1$ over $\C$ and $e\in\mathcal{O}_{\rm min}$ then
$(\mf{g}^e)^*_{\rm sing}$ has codimension $\ge 3$ in $(\mf{g}^e)^*$.
\end{Th}
In conjunction with \cite[Theorem~04]{PanPreYak06} 
this shows that our main result (Theorem~\ref{Th:Main}) is applicable to the minimal nilpotent centralizers $\mf{g}^e$ outside type ${\rm E}_8$. 

\begin{proof}[Proof of  Theorem~\ref{Th:codim3}].
Let $\{e,h,f\}$ be an $\mf{sl}_2$-triple with $e\in\mathcal{O}_{\rm min}$. 
The action of ${\rm ad}\,h$ on $\mf{g}$ gives rise to a
$\Z$-grading $$\mf{g}=\mf{g}(-2)\oplus\mf{g}(-1)\oplus\mf{g}(0)\oplus\mf{g}(1)
\oplus\mf{g}(2)$$ such that $\mf{g}(-2)=\C f$ and $\mf{g}(2)=\C e$. 
Also, $\mf{g}^e=\mf{g}^e(0)\oplus \mf{g}^e(1)\oplus \mf{g}^e(2)$ where $\mf{g}^e(0)$ is an ideal of codimension $1$ in $\mf{g}(0)$. We may assume that our symmetric invariant bilinear form $\kappa$ on $\mf{g}$ 
has the property that $\kappa(e,f)=1$. Then
$[x,y]=\langle x,y\rangle e$ for all
$x,y\in\mf{g}(1)$, where 
$ \langle x,y\rangle=\kappa(f,[x,y])$. Since $\mf{g}^f\cap \mf{g}(1)=\{0\}$ by the $\mf{sl}_2$-theory, the skew-symmetric form
$\langle \cdot\,,\,\cdot\rangle$ on $\mf{g}(1)$ 
is non-degenerate and 
$\mf{g}(1)\oplus \mf{g}(2)$ is isomorphic to a Heisenberg Lie algebra. 

The centralizer $G_0:=Z_G(h)$ is a Levi subgroup of $G$  
with ${\rm Lie}(G_0)=\mf{g}(0)$ and one can choose a maximal torus $T$ in $G_0$ 
and a basis of simple roots $\Pi$ in the root system 
$\Phi$ of $G$ with respect to $T$
in such a way that $e\in\mf{g}_\theta$ where $\theta$ is the highest root of the positive system $\Phi^+(\Pi)$. 
Thanks to Yakimova's result mentioned earlier we may
assume that $G$ is not of type ${\rm A}$. In this case it is well known (and easily seen) that 
$\mf{g}^e(0)=[\mf{g}(0),\mf{g}(0)]$ is a semisimple group and $\mf{g}(1)$ is an irreducible $\mf{g}^e(0)$-module. 
 
The bilinear form $\kappa$ enables us to identify the coadjoint $\mf{g}^e$-module $(\mf{g}^e)^*$
with the factor-module $\mf{g}/\big(\C h\oplus \mf{g}(1)\oplus \mf{g}(2)\big)$. Since the latter identifies with 
$\mf{g}(-2)\oplus\mf{g}(-1)\oplus \mf{g}^e(0)$ as vector spaces and $G_0$-modules, we shall often express linear functions $\chi\in(\mf{g}^e)^*$ in the form $$\chi=\alpha_\chi f+n_\chi+h_\chi \ \mbox{ with }\ \alpha_\chi\in\C,\,\, n_\chi\in\mf{g}(-1) \mbox{ and }\, h_\chi\in\mf{g}^e(0).$$ 
By \cite[p.~366]{PanPreYak06}, the hyperplane $\mathcal{H}:=\mf{g}(-1)\oplus \mf{g}^e(0)$ of $(\mf{g}^e)^*$ given by
the equation $\alpha_\chi=0$ is not contained in $(\mf{g}^e)^*_{\rm sing}$. 

Suppose for a contradiction that $(\mf{g}^e)^*_{\rm sing}$ contains an irreducible 
component $X$ of codimension $\le 2$ in $(\mf{g}^e)^*$.  
Let $N$ be the unipotent radical of the centralizer $G^e$.
If $X\not\subseteq \mathcal{H}$ then the coadjoint action of $N$ on the
the nonempty principal Zariski open subset
$X^\circ:=\{\chi\in X\,|\,\,\chi(e)\ne 0\}$ induces an isomorphism of affine algebraic varieties
$$X^\circ\,\cong\,(N/(N,N))\times 
\big(X^\circ \cap{\rm Ann}\,\mf{g}(1)\big);$$
see \cite[p.~366]{PanPreYak06} for detail. 
Identifying $(\mf{g}^e)^*$ with $\mf{g}/(\C h\oplus \mf{g}(1)\oplus \mf{g}(2)\big)$ we have that
${\rm Ann}\,\mf{g}(1)\,=\,
\C f+\mf{g}^e(0)$ and $X^\circ\cap {\rm Ann}\,\mf{g}(1)\,=\,\C^\times f\oplus X^\circ(0)$ for some nonempty 
locally closed subset $X^\circ(0)$ of $\mf{g}^e(0)$.
Since $\dim N/(N,N)=\dim \mf{g}^e(1)$, this subset  
has codimension $\ge 2$ in $\mf{g}^e(0)$.
Consequently, $X^\circ(0)$ contains a regular element of 
the semisimple Lie algebra $\mf{g}^e(0)$, say $r$. But if $\chi\in X$ is such that $\alpha_\chi\ne 0$, $n_\chi=0$ and $h_\chi=r$ then $(\mf{g}^e)^\chi$ consists of all $x_0+\lambda e$ with $\lambda\in\C$ and $x_0\in\mf{g}^e(0)$ such that $[x_0,r]=0$. So the stabiliser $(\mf{g}^e)^\chi$ has dimension
$1+(\ell-1)=\ell$. This contradiction shows that each  $\chi\in X$ vanishes on $e$, i.e. $X$ is a proper subset of $\mathcal{H}$. As we are assuming that $ (\mf{g}^e)^*_{\rm sing}$ has codimension $\le 2$ in $(\mf{g}^e)^*$ this means that $X$ is a hypersurface in $\mathcal{H}$.

Let $\pi\colon X\rightarrow \mf{g}(-1)$ denote the morphism induced by the canonical projection
$\mathcal{H}\twoheadrightarrow\, \mf{g}(-1)$. 
The theorem on fiber dimensions of a morphism then implies that $\pi$ is either dominant or
the Zariski closure of $\pi(\mathcal{H}_{\rm sing})$ has codimension $1$ in $\mf{g}(-1)$. Moreover, in the latter case the generic fibers of $\pi$ project onto $\mf{g}^e(0)$ while in the former case they project onto hypersurfaces in $\mf{g}^e(0)$.
Since $G_0$ is a connected group and $\mathcal{H}_{\rm sing}$ is $G_0$-stable, the sets $X$ and $\pi(X)$ are $G_0$-stable, too.

If $\chi=n_\chi+h_\chi\in \mathcal{H}$ then the stabilizer 
of $\chi$ in $\mf{g}^e$ consists of all
$x=x_0+x_1+x_2$ with $x_i\in\mf{g}^e(i)$ such that
\begin{equation}\label{stab}
x_0\in\mf{g}^{n_\chi}(0)\ \,\mbox{ and }\ \,[x_1,n_\chi]+
[x_0,h_\chi]\in \C h.
\end{equation}

We first consider the case where $\fing$ is not isomorphic to a Lie algebra of type $\rm C$ (in particular, we exclude the case where $\fing$ is of type ${\rm B}_2$). Applying \cite[Theorem~4.2]{Pan03} or analysing the Dynkin labels of nilpotent $G$-orbits as presented in
\cite[pp.~394--407]{Car85} one observes that in this case the cocharacter $2\theta^\vee\colon \C^\times \rightarrow T$ is optimal in the sense of the Kempf--Rousseau theory for
a nonempty Zariski open subset of $\mf{g}(1)$. It follows that there exists an $\mathfrak{sl}_2$-triple $\{\tilde{e},2h,\tilde{f}\}\subset \mf{g}$ such that
$\tilde{f}\in \mf{g}(-1)$, $\tilde{e}\in \mf{g}(1)$ and the adjoint $G_0$ orbits of $\tilde{e}$ and $\tilde{f}$ are Zariski open in $\mf{g}(1)$ and $\mf{g}(-1)$, respectively.

\medskip

\noindent (A) Suppose $\pi\colon X\rightarrow \mf{g}(-1)$ is a dominant morphism. Then $\pi(X)$ contains a Richardson element 
of the parabolic subalgebra $\mf{g}(0)\oplus \mf{g}(-1)\oplus \mf{g}(-2)$ of $\mf{g}$, say $\tilde{f}$. 
Since all such elements in $\mf{g}(-1)$ are conjugate under the adjoint action of $G_0$, our earlier remark shows that there exists $\tilde{e}\in\mf{g}(1)$ such that $\{\tilde{e}, 2h, \tilde{f}\}$ is an $\mf{sl}_2$-triple in $\mf{g}$. By the $\mf{sl}_2$-theory, the Lie algebra $\mf{g}^{\tilde{f}}(0)=\mf{g}^{\tilde{f}}\cap \mf{g}^h$ is reductive and 
\begin{equation}\label{sum}
\mf{g}(0)=\mf{g}^{\tilde{f}}(0)\oplus [\tilde{f},\mf{g}(1)].
\end{equation} Since non-isomorphic irreducible 
$\mathfrak{sl}_2$-submodules of $\mf{g}$ are orthogonal to each other with respect to $\kappa$, it must be that
$\kappa(z,h)=0$ for all $z\in \mf{g}^{\tilde{f}}(0)$. From this it follows that $\mf{g}^{\tilde{f}}(0)\subseteq [\mf{g}(0),\mf{g}(0)]=\mf{g}^e(0)$.

Our standing hypothesis yields that $X$ contains an element $\chi$ with $n_\chi=\tilde{f}$ and  
$\pi^{-1}(\tilde{f})=\tilde{f}+Y$ for some hypersurface
$Y$ in $\mf{g}^e(0)$. 
Applying to $\chi$ a suitable automorphism $\exp({\rm ad}^*\,y)\in {\rm Ad}^*\,G^e$ with 
$y\in \mf{g}(1)$ and taking (\ref{sum}) into account we may assume further that 
$h_\chi\in \mf{g}^{\tilde{f}}(0)$. 
This entails that $Y$ intersects with the subalgebra $\mf{g}^{\tilde{f}}(0)$ of $\mf{g}^e(0)$. By the affine dimension theorem, all irreducible components of that intersection have codimension $\le 1$ in $\mf{g}^{\tilde{f}}(0)$.

The structure of the Lie algebra $\mf{g}^{\tilde{f}}(0)$
is described in \cite[Tables~2 and 3]{PanPreYak06}. If $\mf{g}$ has type other than ${\rm G}_2$, ${\rm B}_3$ or ${\rm D}_4$ then the semisimple rank of $\mf{g}^{\tilde{f}}(0)$
is $\ge 1$ and in the three excluded cases $\mf{g}^{\tilde{f}}(0)$ is a torus of dimension $0$, $1$ and $2$, respectively. In any event, $Y\cap \mf{g}^{\tilde{f}}(0)$ contains a regular element of $\mf{g}^{\tilde{f}}(0)$; we call it $r'$ (if $\mf{g}$ has type ${\rm G}_2$ then necessarily $r'=0$). 

Taking $\chi\in X$ with $n_\chi=\tilde{f}$ and $h_\chi=r'$ and using (\ref{stab}) it is straightforward to see
that $x_0+x_1+x_2\in (\mf{g}^e)^\chi$ if and only if $x_0$ lies in the centralizer of $r'$ in $\mf{g}^{\tilde{f}}(0)$ and $[x_1,\tilde{f}]\in\C h$. Since the map ${\rm ad}\,\tilde{f}\colon \mf{g}(1)\rightarrow \mf{g}(0)$ is injective, this gives
$x_1\in\C \tilde{e}$.
Since $\mf{g}^{\tilde{f}}(0)$ has rank $\ell -2$ in all cases (by \cite[3.9]{PanPreYak06}) and $x_2\in \C e$ we now conclude that the stabilizer of $\chi$ in $\mf{g}^e$ has dimension $\ell$. As this contradicts out assumption that $\chi\in (\mf{g}^e)^*_{\rm sing}$ we thus deduce that the present case cannot occur.

\medskip

\noindent
(B) Now suppose that the Zariski closure of 
$\pi(X)$ has codimension $1$ in $\mf{g}(-1)$.
This case is more complicated, and we are fortunate that some related work has been done by Panyushev in \cite[Sect.~4]{Pan03}. Given a subspace $V$ in $\mf{g}(i)$ we denote by $V^\perp(-i)$ the orthogonal complement of $V$ in $\mf{g}(-i)$ with respect to $\kappa$. Since $2\theta^\vee$ is an optimal cocharacter
for a nonempty Zariski open subset of $\mf{g}(-1)$ the coordinate ring  $\C[\mf{g}(-1)]$ contains a homogeneous 
function of positive degree which is semi-invariant under the adjoint action of $G_0$; we call it $\varphi$. Since $G_0$ acts on $\mf{g}(-1)$ with finitely many 
orbits, the zero locus of $\varphi$ contains an open
$G_0$-orbit $\mathcal{O}'_0$. This orbit is actually unique by  
\cite[Theorem~4.2(ii)]{Pan03}. We pick an element $f'\in\mathcal{O}'_0$.
Since $[\mf{g}(0),f']$ has codimension $1$ in
$\mf{g}(-1)$ by our choice of $e'$ there is a nonzero
$u\in \mf{g}^{f'}(1)$ such that $\mf{g}^{f'}(1)=[\mf{g}(0),f']^\perp(1)=\C u$. Since $\mf{g}(-1)\oplus \mf{g}(-2)$ is a Heisenberg Lie algebra, $\mf{g}^{f'}(-1)$ has codimension $1$ in $\mf{g}(-1)$. Also, $\dim 
\mf{g}^{f'}(0)=\dim \mf{g}(0)-\dim \mf{g}(-1) +1$ and 
$\mf{g}^{f'}(2)=\{0\}$. From this it is immediate that 
$\mf{g}^{f'}=\mf{g}(-2)\oplus \mf{g}^{f'}(-1)\oplus \mf{g}^{f'}(0)\oplus \C u$ and 
$$\dim \mf{g}^{f'}= \dim \mf{g}(0)+2=(\dim\mf{g}^{\tilde{f}})+2.$$ 

Since the semisimple element $h$ normalizes the line $\C f'$ it lies in a Cartan subalgebra of $\mf{g}$ contained in the optimal parabolic subalgebra of the 
$G$-unstable vector $e'\in\mf{g}$. By
the Kempf--Rousseau theory, this Cartan subalgebra contains an element $h'$ with the property  that $\{e',h',f'\}$ is an $\mf{sl}_2$-triple of $\mf{g}$.
Since $[h,h']=0$ we have that $h'\in\mf{g}(0)$.
Writing $e'=\textstyle{\sum}_i e'_i$ with $e'_i\in\mf{g}(i)$ and taking into account the fact that $[h', e']=2e'$ and $[h,\mf{g}(i)]\subseteq \mf{g}(i)$ for all $i$, one observes that
$[h',e'_1]=2e_1'$. As $[e',f']=h=[e'_1,f']$ it follows that
$\{e_1',h',f'\}$ is an $\mf{sl}_2$-triple of $\mf{g}$. So we may assume without loss that
$e'\in\mf{g}(1)$. 

For $i,j\in\Z$ define $\mf{g}(i,j)=\{x\in\mf{g}(j)\,|\,\,
[h',x]=ix\}$. By the 
$\mf{sl}_2$-theory, the Lie algebra $\mf{g}$ thus acquires a bi-grading $\mf{g}=\bigoplus_{i,j\in\Z}\,\mf{g}(i,j)$.  By 
\cite[4.6]{Pan03}, it endows $\mf{g}$ with a symmetry of type ${\rm G}_2$. Specifically, the set $\Sigma'$ of all $(i,j)\ne (0,0)$ with $\mf{g}(i,j)\ne \{0\}$ forms a root system of type ${\rm G}_2$ in
$\mathbb{R}^2$. More precisely,
$\mf{g}(\pm 2)=\mf{g}(\pm 3,\pm 2)$,
$\mf{g}(0)=\mf{g}(-1,0)\oplus\mf{g}(0,0)\oplus \mf{g}(1,0)$, and
$$\mf{g}(\pm 1)\,=\,\mf{g}(\pm 3,\pm 1)\oplus \mf{g}(\pm 2,\pm 1)\oplus \mf{g}(\pm 1,\pm 1)\oplus \mf{g}(0,\pm 1).$$
Furthermore, the subspaces $\mf{g}(i,j)$ corresponding to the long roots of $\Sigma'$ have dimension $1$ whereas all $\mf{g}(i,j)$'s labelled by the short roots of $\Sigma'$ have the same dimension $a$ depending on the type of $\mf{g}$. We shall see later that $a+3=h^\vee$, the dual Coxeter number of $\mf{g}$.

Note that $\mf{g}(0,1)=\mf{g}^{f'}(1)=\C u$ and $\mf{g}(0,-1)=\mf{g}^{e'}(-1)$. Also, $$\mf{g}^{f'}(0)\cap \mf{g}^{e'}(0)=\mf{g}^{f'}(0,0)=\mf{g}^{e'}(0,0)\ \mbox{ and } 
\ \mf{g}(0)=\mf{g}^{f'}(0,0)\oplus [e',\mf{g}(-2,-1)]$$
by the $\mf{sl}_2$-theory. We set $\mf{k}:=\mf{g}^{f'}(0,0)$ and $\mf{p}:=[e',\mf{g}(-2,-1)]$.
As $({\rm ad}\,e')^3$ annihilates the abelian subalgebra $\mf{g}(-2,-1))$ of $\mf{g}$, we have that 
\begin{eqnarray*}
0&=&\textstyle{\frac{1}{3}}({\rm ad}\, e')^3([x,y])=[({\rm ad}\, e')^2(x),({\rm ad}\, e')(y)]+[({\rm ad}\, e')(x),({\rm ad}\, e')^2(y)]\\
&=&\big[e',[[e',x],[e',y]]\big]=[e',[\mf{p},\mf{p}]]
\end{eqnarray*}
for all $x,y\in\mf{g}(-2,-1)$. Consequently, $[\mf{p},\mf{p}]\subseteq \mf{k}$ implying that $\mf{g}(0)=\mf{k}\oplus\mf{p}$ is a symmetric decomposition of the Levi subalgebra $\mf{g}(0)$ of $\mf{g}$. Note that the restriction of $\kappa$ to $\mf{k}$ is non-degenerate and 
$\mf{p}=[e',\mf{g}(-2,-1)]=\mf{k}^\perp=[f',\mf{g}(2,1)]$ by dimension reasons. Since
${\rm ad}\,f'$ is injective on $\mf{g}(1,0)$ by the $\mf{sl}_2$-theory and
$[f',\mf{g}(-1,0)]=\mf{g}(-3,-1)$ is $1$-dimensional, the
ideal $\mf{g}^{f'}(-1,0)$ of $\mf{g}^{f'}(0)$ has codimension $1$ in $\mf{g}(-1,0)$ and
\begin{equation}\label{0-cen} 
\mf{g}^{f'}(0)=\mf{k}\oplus \mf{g}^{f'}(-1,0).
\end{equation}

Recall that in the present case we can take $\chi\in X$ with $n_\chi=f'$ while our choice of 
$h_\chi\in \mf{g}^e(0)$ is unconstrained. Since $\mf{g}(3,1)$ is spanned by $({\rm ad}\,e')^3(f)$, our earlier remarks in this part yield 
$$[f',\mf{g}(1)]=\C [e',[e',f]]\oplus\mf{g}(-1,0)\oplus\mf{p}.$$  
Applying to $\chi=f'+h_\chi$ a suitable automorphism $\exp({\rm ad}^*\,y)\in {\rm Ad}^*\,G^e$ with $y\in \mf{g}(1)$ we may thus assume that $h_\chi=h_{\chi,0}+h_{\chi,1}$ for some $h_{\chi,0}\in\mf{k}$ and $h_{\chi,1}\in 
\mf{g}^{e'}(1,0)$. Note that $\kappa$ gives rise to 
a perfect pairing between $\mf{g}^{f'}(0)=\mf{k}\oplus \mf{g}^{f'}(-1,0)$ and $\mf{g}^{e'}(0)=\mf{k}\oplus \mf{g}^{e'}(1,0)$.

Set $\mf{m}:=\mf{k}^e\oplus\mf{g}^{f'}(-1,0)=\mf{g}^{f'}(0)\cap\mf{g}^e(0)$, an ideal of codimension $1$
in $\mf{g}^{f'}(0)$, and denote by $\bar{\chi}$ the linear function on $\mf{m}$ given by $\bar{\chi}(y)=\kappa(h_\chi,y)$ for all $y\in \mf{m}$.
Our preceding remarks imply that $h_\chi$ can be chosen in such a way that $\dim (\mf{m}^*)^{\bar{\chi}}={\rm ind}\,\mf{m}$ where, as before, ${\rm ind\,}\mf{m}$ denotes the index of the Lie algebra $\mf{m}$.
If $x=x_0+x_1+x_2$ with $x_i\in\mf{g}^e(i)$ lies $(\mf{g}^e)^\chi$ then
(\ref{stab}) yields that $x_0=x_{0,0}+x_{-1,0}$ for some $x_{0,0}\in 
\mf{k}^e$ and $x_{1,-0}\in  \mf{g}^{f'}(-1)$.
Moreover,
$$[h_{\chi,0},x_{0,0}]+[h_{\chi,1},x_{-1,0}]+
[h_{\chi,0},x_{-1,0}]+[h_{\chi,1},x_{0,0}]+[f',x_1]\in \C h.$$ Then 
$\kappa([h_{\chi,0}+h_{\chi,1},x_{0,0}+
x_{0,-1}],y_{0,0}+y_{0,-1})=0$ for all
$y_{0,0}\in\mf{k}^e$ and $y_{-1,0}\in\mf{g}^{f'}(-1,0)$ forcing $x_0=x_{0,0}+x_{0,-1}\in\mf{m}^{\bar{\chi}}$. Since $\mf{g}(-1,0)\subset [f',\mf{g}(1)]$, this also shows that for any $x_0\in\mf{m}^{\bar{\chi}}$ there is $x_1\in\mf{g}(1)$ such that $x_0+x_1\in(\mf{g}^e)^\chi$.
As a consequence, the canonical projection $\mf{g}^e\twoheadrightarrow \mf{g}^e(0)$ gives rise to a surjective Lie algebra homomorphism $\psi\colon(\mf{g}^e)^\chi\twoheadrightarrow \mf{m}^{\bar{\chi}}$ whose kernel consists of all $x_1+x_2$ with $x_i\in\mf{g}(i)$ and $[f',x_1]\in\C h$.

Let $K = 2h^\vee \kappa$ be the Killing form of $\mf{g}$. It is immediate from our earlier remarks that $4h^\vee=K(h,h)=2(2a+2)+8=4(a+3)$ and $K(h,h')=2(3a+3)+12=6(a+3)=6h^\vee$
(here $h^\vee$ is the dual Coxeter number of $\mf{g}$).
Therefore, $\kappa(h',h)=3$. Similarly, 
$K(h',h')=2(5a+9)+2a+18=12(a+3)$, so that $\kappa(h',h')=6$.
If $h=[f',v]$ for some $v\in\mf{g}$ then
$\kappa(h,h)=\kappa([h,f'],v)=-\kappa(f',v)$ and $\kappa(h',h)=\kappa([h',f'],v])=-2\kappa(f',v)$ implying $\kappa(h',h)=4$. This contradiction shows that
${\rm Ker}\,\psi=\mf{g}^{f'}(1)\oplus \mf{g}(2)=\C u\oplus \C e$ is $2$-dimensional. As a result,
\begin{equation}\label{ind}
\dim(\mf{g}^e)^\chi={\rm ind}\,\mf{m}+2.
\end{equation}
Next we observe that $f'$ lies in $\mf{g}^{h'-2h}=\mf{g}(-2,-1)\oplus\mf{g}(0,0)\oplus\mf{g}(2,1)$, a Levi subalgebra of $\mf{g}$. Let $\mf{l}$ denote the orthogonal complement of $h'-2h$ in $\mf{g}^{h'-2h}$. 
As $\kappa(h'-2h,h'-2h)=6-12+8=2$ we have that $
\mf{g}^{h'-2h}=\mf{l}\oplus \C (h'-2h)$. In particular,
$\mf{l}$ is a reductive subalgebra of $\mf{g}$ of rank $\ell-1$.
In order to compute the index of $\mf{m}$ in a case-free fashion we shall identify $\mf{m}$ with a subalgebra of $\mf{l}^{f'}$. 

We first recall tat $\mf{g}^{f'}(-1,0)$ is an abelian ideal of $\mf{m}$, so that $\mf{m}\cong 
\mf{k}^e\ltimes \mf{g}^{f'}(-1,0)$ as Lie algebras. Since $\mf{k}^e=\mf{k}\cap [\mf{g}(0),\mf{g}(0)]$ is orthogonal to both $h'$ and $h$ by the $\mf{sl}_2$-theory, it is a subalgebra of $\mf{l}^{f'}(0,0)$. From this it is immediate that $\mf{l}=\mf{k}^e\oplus \mf{g}(-2,-1).$ Since 
$[f,[e',\mf{g}(-1,0)]]=\mf{g}(-2,-1)$ by the $\mf{sl}_2$-theory and both $f$ and $e'$ commute with $\mf{k}^e$, the $\mf{k}^e$-modules $\mf{g}(-1,0)$ and $\mf{g}(-2,-1)$ are isomorphic. If $f'\in [f,[e',\mf{g}^{f'}(-1,0)]]$ then
$$3=\textstyle{\frac{1}{2}}\kappa(h',h')=\kappa(e',f')=
\kappa(e',[f,[e',w]])=-\kappa([e',[e',f]],w)$$ for some $w\in\mf{g}^{f'}$. But since $[h',e]=3e$ and $[e',e]=0$
the $\mf{sl}_2$-theory also yields $[e',[e',f]]\in {\rm Im}\,{\rm ad}\,f'$. Then $\kappa([e',[e',f]],w)=0$. This contradiction shows that
$\mf{g}(-2,-1)=[f,[e',\mf{g}^{f'}(-1,0)]]\oplus \C f'$ as $\mf{k}^e$-modules and $$\mf{m}\cong \mf{k}^e\ltimes 
\mf{g}^{f'}(-1,0)\cong \mf{k}^e\ltimes [f,[e',\mf{g}^{f'}(-1,0)]]$$ identifies with an ideal of codimansion $1$ in $\mf{l}^{f'}$ complementary to $\C f'$; we call it $\mf{m}'$. Since ${\rm ind}\,\mf{l}^{f'}={\rm rk}\,\mf{l}=\ell-1$ by \cite[Theorem~3.5]{Pan03}, for example, and $\mf{l}^{f'}=\mf{m}'\oplus \C f'$, we now deduce that ${\rm ind}\,\mf{m}={\rm ind}\,\mf{m}'=\ell-2$. But then
(\ref{ind}) gives $\dim(\mf{g}^e)^\chi=\ell$ contrary to our choice of $\chi$. This contradiction proves the theorem for all simple Lie algebras of type other than $\rm C$.

\medskip

\noindent
(C) Finally, let $G$ be of type ${\rm C}_\ell$ and suppose first that $\ell=2$. In this case it follows from the main results of
\cite{PanPreYak06} that the $\C$-algebra $S(\mf{g}^e)^{\mf{g}^e}$ is freely generated by two homogeneous invariants $^{e}\!F_1$ and $^{e}\!F_2$ of degree $1$ and $3$, and $(\mf{g}^e)^*_{\rm sing}$ coincides with their Jacobian locus $\mathcal{J}(^{e}\!F_1, ^{e}\!F_2)$. The centraliser $\mf{g}^e$ has a $\C$-basis $\{E,H,F,u,v,e\}$ such that $\{E,H,F\}$ is an $\mf{sl}_2$-triple in $\mf{g}^e(0)$ and $u,v\in \mf{g}^e(1)$ have the property that 
$[E,v]=u$, $[F,u]=v$ and $[u,v]=e$. Using this basis it is not hard to describe the invariant ring $S(\mf{g}^e)^{\mf{g}^e}$ explicitly. Namely, we may assume without loss of generality that $^{e}\!F_1=e$ and  
$$^{e}\!F_2=(4EF+H^2)e+2Ev^2+2uvH-2Fu^2.$$ As $\deg(^{e}\!F_1)=1$, it is then straightforward to see that $\mathcal{J}(^{e}\!F_1, ^{e}\!F_2)$ is isomorphic to the intersections of five quadrics 
$$2yt+q^2=0,\ \, 2xt-p^2=0,\ \,qz-2yp=0,\ \,2xq+pz=0,\
\,  zt+pq=0$$ in the affine space $\mathbb{A}^6$ with coordinates $x,y,z, p,q, t$. This variety is nothing but
$$\{(p^2/2t,-q^2/2t,-pq/t,p,q,t)\,|\,\,p,q\in\C,\,t\in\C^\times\}\sqcup\{(x,y,z,0,0,0)\,|\,\,x,y,z\in\C\}.$$ We thus deduce that  $(\mf{g}^e)^*_{\rm sing}$
has two irreducible components both of which are $3$-dimensional. So the statement holds for $\ell=2$.

Now suppose $\ell\ge 3$. At the beginning of this proof we have assumed for a contardiction that an irreducible hypersurface $X$ of $\mathcal{H}$ is contained in $(\mf{g}^e)^*_{\rm sing}$. Recall that $N$ denotes the unipotent radical of $G^e$ and
$\pi\colon\, X\rightarrow \mf{g}(-1)$ is the map induced by the second projection $\mf{g}^e(0)\oplus\mf{g}(-1)\twoheadrightarrow \mf{g}(-1)$.
Since in the present case the derived subgroup of $G_0$ acts transitively on the nonzero vectors of the $2(\ell-1)$-dimensional vector space $\mf{g}(-1)$
the morphism $\pi$ is surjective and $\dim \pi^{-1}(v)=(\dim X)-2(\ell-1)$ for any nonzero $v\in \mf{g}(-1)$. 

Using the standard realisation of the root system $\Phi$ we may assume that $\theta=2\varepsilon_1$ and the Lie algebra $\mf{g}^e(0)$ is generated 
by all root vectors $e_\gamma$ and $f_\gamma$, where $\gamma =\varepsilon_i\pm\varepsilon_j$ and $2\le i<j\le \ell$. We may (and will) take for $v$ a short root vector 
$f_{\varepsilon_1-\varepsilon_2}$. 
Set $\mf{l}=\mf{g}^e(0)$, a simple Lie algebra of type
${\rm C}_{\ell-1}$. The adjoint action of $h_v:=h_{\varepsilon_2-\varepsilon_1}$ on $\mf{l}$
gives rise to  a $\Z$-grading $$\mf{l}=\mf{l}(-2)\oplus\mf{l}(-1)\oplus\mf{l}(0)\oplus\mf{l}(1)\oplus\mf{l}(2)$$ such that $\mf{l}(2)=\C e_{2\varepsilon_2}$ and $[v,\mf{g}(1)]=\C h_v\oplus \mf{l}(1)\oplus\mf{l}(2)$. As $h_{\varepsilon_2-\varepsilon_1}-h_{2\varepsilon_2}=-h$, the same grading of $\mf{l}$ is induced by the adjoint action of $h_{2\varepsilon_2}$. More importantly,
$v=f_{\varepsilon_1-\varepsilon_2}$ and $e_0:=e_{2\varepsilon_2}$ have the same centralizer in $\mf{l}$, namely, $[\mf{l}(0),\mf{l}(0)]\oplus\mf{l}(1)\oplus\mf{l}(2)$.

Let $S$ be the set of all linear functions on $\mf{g}^e$
that vanish on $\mf{g}(2)\oplus\mf{g}(1)\oplus\mf{l}(-2)\oplus\mf{l}(-1)\oplus\C h_{2\varepsilon_2}$. This subspace of $(\mf{g}^e)^*$ is canonically isomorphic to 
the dual space of the minimal nilpotent centralizer $\mf{l}^{e_0}$ of type ${\rm C}_{\ell-1}$, Our discussion in the previous paragraph implies that the coadjoint action of $N$ gives rise to an isomorphism 
$$\pi^{-1}(v)\,\cong\,\big(N/(N,N)\big)\times \big(v+
V\cap X\big)$$ of affine algebraic varieties. As a consequence,
\begin{eqnarray*}
\dim(S\cap X)&=&\dim \pi^{-1}(v)-2(\ell-1) =\dim X-4(\ell-1)\\
&=&\dim \mf{g}^e-4(\ell-1)-2=\dim\mf{l}-2(\ell-2)-3\\
&=&\dim \mf{l}(0)+\dim\mf{l}(1)+\dim\mf{l}(2)-2
\,=\,\dim \mf{l}^{e_0}-1.
\end{eqnarray*}
Therefore, $S\cap X$ has codimension $1$ in $(\mf{g}^{e_0})^*$. By \cite[Theorem~04]{PanPreYak06}, there exists $\xi\in S\cap X$  with
$\dim (\mf{l}^{e_0})^{\bar{\xi}}=\ell-1$, where
$\bar{\xi}$ is the restriction of $\xi$ to
$\mf{l}^{e_0}$
(here we use the fact that $\ell\ge 3$). 
Since $\mf{l}^{e_0}=\mf{l}^v$ it follows
from the definition of $V$ and the preceding discussion that $x=x_0+x_1+x_2$ with $x_i\in\mf{g}^e(i)$ lies in
$(\mf{g}^e)^\xi$ if and only if $x_0\in (\mf{l}^{e_0})^{\bar{\xi}}$ and $x_1=0$.
But then $\dim (\mf{g}^e)^\xi=(\ell-1)+1=\ell$ violating our assumption that $X\subseteq (\mf{g}^e)^*_{\rm sing}$.
This contardiction completes the proof of Theorem~\ref{Th:codim3}.
\end{proof}
 \begin{Rem}
 If $\mf{g}$ has type other than ${\rm A}$ or ${\rm E}_8$ then combining 
 Theorem~\ref{Th:codim3} with \cite[Theorem~04]{PanPreYak06} and \cite[Proposition~1.6]{OoVdB10} one obtains that {\it all} irreducible components of $(\mf{g}^e)^*_{\rm sing}$ have codimension $3$ in $(\mf{g}^e)^*$. We stress that \cite[Proposition~1.6]{OoVdB10} applies because outside type ${\rm A}$ the Lie algebra $\mf{g}^e$ is perfect and hence all 
 semi-invariants of $S(\mf{g}^e)$ under the action of $\mf{g}^e$ lie in $S(\mf{g}^e)^{\mf{g}^e}$.
If $\mf{g}$ is of type ${\rm E}_8$ then \cite[Theorem~04]{PanPreYak06} and \cite[Proposition~1.6]{OoVdB10} are no longer applicable,
but according to an unpublished result of Yakimova it is still true that
 $(\mf{g}^e)^*_{\rm sing}$ has codimension $3$ in
$(\mf{g}^e)^*$.
\end{Rem}
\section{Some explicit formulae in type $A$: the minimal nilpotent case}
\label{sec:example}
Let $\fing=\mf{gl}_n$,
and let $e_{ij}$ be a standard 
basis element of $\fing$.
Let $\affg_-\+\C \tau$ be the extended Lie algebra
with $\tau$ acting on $\affg$ by the rule
\begin{align*}
 [\tau,x_{(-m)}]=mx_{(-m)}.
\end{align*}
For an $n\times n$ matrix $A=(a_{ij})$ with entries in an associative ring,
denote by
$\on{cdet}A$  the column determinant
defined by the formula
\begin{align*}
 \on{cdet}A=\sum_{\sigma\in\mf{S}_n}\on{sgn}(\sigma)a_{\sigma(1)1}\dots a_{\sigma(n)n}.
\end{align*}
Let  $ \tau+E_{(-1)}$ be the $n\times n$ matrix with entries in
$ \affg_-\*\C\tau$ given by
\begin{align*}
 \tau+E_{(-1)}
=
\begin{pmatrix}
 \tau+(e_{11})_{-1}&(e_{12})_{(-1)}&\dots &(e_{1n})_{(-1)}\\
(e_{21})_{-1}& \tau+(e_{22})_{(-1)}&\dots &(e_{2n})_{(-1)}\\
\vdots &\cdots &\ddots &\vdots \\
(e_{n1})_{-1}& (e_{n2})_{(-1)}&\dots &\tau+(e_{nn})_{(-1)}
\end{pmatrix}.
\end{align*}

Let $e=e_{n,n-1}$, a
 minimal nilpotent element of $\fing$, and put
$$\fing_2=\on{span}\{e_{n,i}; 1\leq i\leq n-1\},\ \ 
\fing_{-2}=\on{span}\{e_{i,n}; 1\leq i\leq n-1\},$$ and
$\fing_0=\on{span}\{e_{ij};1\leq i,j\leq n-1\}
+\on{span}\{e_{n,n}\}$,
Then 
\begin{align*}
\fing=\fing_{-2}\+\fing\+\fing_{2}
\end{align*}
 defines a good even 
grading for $e$. In particular,  $\fing^e=\fing_0^e\oplus \fing_2$.
Denote by $\kappa_{e,c}$ the
invariant symmetric bilinear form 
on $\fing^e$
such that
\begin{align*}
 \kappa_{e,c}(x,y)=\begin{cases}
-\on{tr}_{\fing_0}(\ad x \ad y)&\text{for }x,y\in \fing_0^e,\\
0
&\text{else. }\\
		   \end{cases}
\end{align*}
Then $\kappa_{e,c}$ coincides the form
defined in Theorem \ref{Th:loop-filtration}
on $\mf{sl}_n^e$.
We have
$V^{\kappa_{e,c}}(\fing^e)=V^{\kappa_{e,c}}(\mf{sl}_n^e)\* \pi$,
where
$\pi$ is the commutative vertex algebra
that is
freely generated by one element 
$I_{(-1)}\mathbb{I}$,
$
 I:=e_{11}+e_{22}+e_{33}+\dots +e_{nn}
$.
Therefore,
\begin{align*}
Z(V^{\kappa_{e,c}}(\fing^e))=Z(V^{\kappa_{e,c}}(\mf{sl}_n))\* \pi
\end{align*}
and
 there is no harm in replacing $\mf{sl}_n$ with $\fing$.

Note that
$\dim \fing^e=\dim \fing_0=(n-1)^2+1$.
We have
\begin{align*}
 \fing^e=\on{span}\{I,\ e_{i,j}\ (1\leq i\leq n-2,
1\leq j\leq n-1),\
e_{n,j}\ (1\leq j\leq n-1)\}.
\end{align*}

Consider the $(n-1)\times (n-1)$ matrix
$Z$ 
obtained from
$\tau+E_{(-1)}$ by removing its
$(n-1)$-th row and $n$-th column. Then
\begin{align*}
 Z=\begin{pmatrix}
    \tau +(e_{11})_{(-1)}
&(e_{12})_{(-1)}&\dots & (e_{1\ n-2})_{(-1)} &(e_{1\ n-1})_{(-1)}
\\ 
(e_{21})_{(-1)} &     \tau +(e_{22})_{(-1)} &\dots &
(e_{2\ n-2})_{(-1)} &(e_{2\ n-1})_{(-1)}\\ 
\vdots & \vdots & \ddots& \vdots  &\vdots \\
(e_{n-2 \ 1})_{(-1)} &(e_{n-2 \ 2})_{(-1)} 
& \dots &    \tau +(e_{n-2\ n-2})_{(-1)}  &(e_{n-2\ n-1})_{(-1)}\\ 
(e_{n~1})_{(-1)} &(e_{n~ 2})_{(-1)} 
& \dots &    (e_{n\ n-2})_{(-1)}  &(e_{n\ n-1})_{(-1)}
   \end{pmatrix}
\end{align*}
The
 entries of $Z$
belong to the subalgebra
$\affg_-^e\+ \C \tau$
of $\affg_-\+ \C \tau$,
and so
its column determinant $\on{cdet}Z$ is an element of $U(\affg^e_-\+ \C
\tau)$. We may also regard it as an element of $V^{\kappa_{e,c}}(\fing^e)\*
\C[\tau]$.

Arguing as in \cite{CheMol09} one proves the following assertion:
\begin{Th} The $\C[T]$-module
 $Z(V^{\kappa_{e,c}}(\fing^e))$ 
is freely generated by
 $I_{(-1)}$ and  
the coefficient $Q_1,\dots, Q_{n-1}$
of the polynomial
\begin{align*}
 \on{cdet}(Z)=Q_1 \tau^{n-2}+Q_2\tau^{n-3}+
\dots Q_{n-2}\tau+Q_{n-1},
\end{align*}
so that
$Z(V^{\kappa_{e,c}}(\fing^e))\,=\,\C[T^j I_{(-1)},
T^j Q_i;
1\leq i\leq n-1, j\in\Z_{\ge 0}]$.
\end{Th}

Note that
$Q_{1}=e_{(-1)}$.

\begin{Ex}
 If $n=3$ then $Q_1=(e_{32})_{(-1)}$ and
\begin{align*}
Q_2=(e_{11})_{(-1)}(e_{32})_{(-1)}-(e_{31})_{(-1)}(e_{12})_{(-1)}
+(e_{32})_{(-2)}.
\end{align*}
If $n=4$ then $Q_1=(e_{43})_{(-1)}$ and
\begin{align*}
& Q_2=\begin{vmatrix}
       (e_{11})_{(-1)} &(e_{13})_{(-1)}\\
       (e_{41})_{(-1)} &(e_{43})_{(-1)}
      \end{vmatrix}
+\begin{vmatrix}
       (e_{22})_{(-1)} &(e_{23})_{(-1)}\\
       (e_{42})_{(-1)} &(e_{43})_{(-1)}
      \end{vmatrix}
+2 (e_{43})_{(-2)},\\
& Q_3=\begin{vmatrix}
       (e_{11})_{(-1)} &(e_{12})_{(-1)}&(e_{13})_{(-1)}
\\ (e_{21})_{(-1)} &(e_{22})_{(-1)}&(e_{23})_{(-1)}\\
       (e_{41})_{(-1)} &(e_{42})_{(-1)} &(e_{43})_{(-1)}
      \end{vmatrix}
+\begin{vmatrix}
       (e_{11})_{(-1)} &(e_{13})_{(-2)}\\
       (e_{41})_{(-1)} &(e_{43})_{(-2)}
      \end{vmatrix}\\
&\qquad +\begin{vmatrix}
       (e_{22})_{(-2)} &(e_{23})_{(-1)}\\
       (e_{42})_{(-2)} &(e_{43})_{(-1)}
      \end{vmatrix}
+\begin{vmatrix}
       (e_{22})_{(-1)} &(e_{23})_{(-2)}\\
       (e_{42})_{(-1)} &(e_{43})_{(-2)}
      \end{vmatrix}
+2(e_{43})_{(-3)}.
\end{align*}
Here we use vertical brackets as short-hand notation for column determinants in non-commutative rings.
\end{Ex}

Let $\chi\in (\fing^e)^*$,
and set $\chi_{ij}:=\chi(e_{ij})$.
Also, put $\partial_u:=\frac{d}{du}$.
Let $A$ be the 
$(n-1)\times (n-1)$ matrix with entries
in $U(\fing^e)\* \C [u^{-1}]\*  \C[ \partial_u]$
obtained by replacing
$(e_{ij})_{(-1)}$ by $u^{-1} e_{ij}+\chi_{ij}$
and $\tau$ by $-\partial_u$:
Its column determinant $\on{cdet}(A)$ is an element of
$U(\fing^e)\*\C[u^{-1}]\* \C[\partial_u]$.
Write
\begin{eqnarray*}
 \on{cdet}(A)&=&A_1 (-\partial_u)^{n-2}+A_2(-\partial_u)^{n-1}
+\dots + A_{n-1}(-\partial_u)+A_{n-1},\\
A_i&=&A_i^{(0)}u^{-i}+A_i^{(1)}u^{1-i}+\dots 
+A_{i}^{(i-1)}u^{-1}+A_{i}^{(i)}
\end{eqnarray*}
and note that
$A_1^{(0)}=e$.

\begin{Th}
 For a regular $\chi\in (\fing^e)^*$ 
we have
\begin{align*}
\mc{A}_{e,\chi}=\C[I,\ A_i^{(j)};i=1,\dots, n-1, j=0,1,\dots,i-1].
\end{align*}
\end{Th}

\begin{Ex}
 If $n=3$ then
\begin{align*}
& A=\begin{vmatrix}
-\partial_u+u^{-1}e_{11}+\chi_{11} &u^{-1}e_{12}+\chi_{12}\\
u^{-1}e_{31}+\chi_{31}&u^{-1}e_{32}+\chi_{32}
\end{vmatrix},\\
&A_1=e_{32}u^{-1}+\chi_{32},\\
&  A_2=\begin{vmatrix}
       e_{11}u^{-1}+\chi_{11}& e_{12}u^{-1}+\chi_{12}\\
e_{31}u^{-1}+\chi_{31}&e_{32}u^{-1}+\chi_{32}
      \end{vmatrix}
+e_{31}u^{-2}.
 \end{align*}
If $n=4$ then
\begin{align*}
& A=\begin{vmatrix}
-\partial_u+u^{-1}e_{11}+\chi_{11} 
&u^{-1}e_{12}+\chi_{12}
&u^{-1}e_{13}+\chi_{13}\\
u^{-1}e_{21}+\chi_{21}
&-\partial_u+u^{-1}e_{22}+\chi_{22}& 
u^{-1}e_{23}+\chi_{23}
\\
u^{-1}e_{41}+\chi_{41}&u^{-1}e_{42}+\chi_{41}
&u^{-1}e_{43}+\chi_{43}
\end{vmatrix},\\
&A_1=e_{43}u^{-1}+\chi_{43},\\
& A_2=\begin{vmatrix}
       e_{11}u^{-1}+\chi_{11} &e_{13}u^{-1}+\chi_{13}\\
       e_{41}u^{-1}+\chi_{41} &e_{43}u^{-1}+\chi_{43}
      \end{vmatrix}
+\begin{vmatrix}
       e_{22}u^{-1}+\chi_{22} &e_{23}u^{-1}+\chi_{23}\\
       e_{42}u^{-1}+\chi_{41} &e_{43}u^{-1}+\chi_{43}
      \end{vmatrix}
+2 e_{43}u^{-2},\\
& A_3=\begin{vmatrix}
       e_{11}u^{-1}+\chi_{11} &e_{12}u^{-1}+\chi_{12}&e_{13}u^{-1}+\chi_{13}
\\ e_{21}u^{-1}+\chi_{21} &e_{22}u^{-1}+\chi_{22}&e_{23}u^{-1}+\chi_{23}\\
       e_{41}u^{-1}+\chi_{41} &e_{42}u^{-1}+\chi_{42} &e_{43}u^{-1}+\chi_{43}
      \end{vmatrix}
+\begin{vmatrix}
       e_{11}u^{-1}+\chi_{11} &e_{13}u^{-2}\\
       e_{41}u^{-1}+\chi_{41} &e_{43}u^{-2}
      \end{vmatrix}\\
&\qquad +\begin{vmatrix}
       e_{22}u^{-2} &e_{23}u^{-1}+\chi_{23}\\
       e_{42}u^{-2} &e_{43}u^{-1}+\chi_{43}
      \end{vmatrix}
+\begin{vmatrix}
       e_{22}u^{-1}+\chi_{22} &e_{23}u^{-2}\\
       e_{42}u^{-1}+\chi_{42} &e_{43}u^{-2}
      \end{vmatrix}
+2e_{43}u^{-3}.
\end{align*}

\end{Ex}

\end{document}